\documentclass[11pt,reqno]{amsart}
\usepackage[T1]{fontenc}
\usepackage{graphicx, amssymb, amsmath}
\usepackage{cite}
\usepackage{color}
\definecolor{MyLinkColor}{rgb}{0,0,0.4}

\newcommand{\R}{{\mathbb R}}

\newcommand{\bA}{{\mathbb{A}}}
\newcommand{\bB}{{\mathbb{B}}}

\newcommand{\N}{{\mathbb N}}

\newcommand{\kH}{\mathcal{H}}
\newcommand{\cO}{\mathcal{O}}

\newcommand{\kL}{\mathcal{L}}

\newcommand{\re}{\mathop{\rm Re}\nolimits}

\newcommand{\PV}{\mathop{\rm PV}\nolimits}
\newcommand{\ov}{\overline}
\newcommand{\p}{\partial}

\newcommand{\e}{\varepsilon}

\newcommand{\0}{\Omega}
\newcommand{\G}{\Gamma}

\newcommand{\supp}{\mathop{\rm supp}\nolimits}
\newtheorem{thm}{Theorem}[section]
\newtheorem{prop}[thm]{Proposition}
\newtheorem{lemma}[thm]{Lemma}

\theoremstyle{remark}

\numberwithin{equation}{section} 
\usepackage[colorlinks=true,linkcolor=MyLinkColor,citecolor=MyLinkColor]{hyperref}

\begin{document}

\title[The Mullins-Sekerka problem via the method of potentials]{The Mullins-Sekerka problem via the method of potentials}

\thanks{}

\author[J. Escher]{Joachim Escher}
\address{Institut f{\"u}r Angewandte Mathematik, Leibniz Universit{\"a}t Hannover, Welfengarten~1, 30167 Hannover, Germany.}
\email{escher@ifam.uni-hannover.de}

\author{Anca--Voichita Matioc}
\author{Bogdan--Vasile Matioc}
\address{Fakult\"at f\"ur Mathematik, Universit\"at Regensburg,   93040 Regensburg, Deutschland.}
\email{anca.matioc@ur.de}
\email{bogdan.matioc@ur.de}

\begin{abstract}
It is shown that the two-dimensional Mullins-Sekerka problem is well-posed in all subcritical Sobolev spaces $H^r(\R)$ with $r\in(3/2,2).$
This is the first result where this issue is established in an unbounded geometry. 
The novelty of our approach is the use of potential theory to formulate the model as an evolution problem with nonlinearities expressed by singular integral operators.
\end{abstract}


\maketitle

\pagestyle{myheadings}
\markboth{\sc{J. Escher, A.-V.~Matioc, \& B.-V.~Matioc}}{\sc{The Mullins-Sekerka problem via the method of potentials}}

 \section{Introduction}\label{Sec:1}

The Mullins-Sekerka problem in a bounded geometry is a moving boundary problem which appears as the gradient flow of the area functional with respect
to a metric which is formally equivalent to the $H^{-1}$-metric on the tangent space of all oriented hypersurfaces which enclose a fixed volume~\cite{ON01, Ga13}. 
It describes the evolution of two domains $\0^+(t)$ and~${\0^-(t)}$  together with the sharp interfaces  $\Gamma(t)$ that separates them 
 in such a way that the volumes of $\0^\pm(t)$  are preserved and the area  of~$\Gamma(t)$ is decreased~\cite{EGK17, Gu93, Ga13}.
Known also as  the two-phase Hele-Shaw problem,  it may  also be derived as  a singular limit of the Cahn-Hilliard problem when the thickness of the transition layer between the phases vanishes \cite{S96, ABC94}.

Most of the mathematical studies regarding this two-phase problem consider a bounded geometry with $\Omega^\pm(t)$ being open subsets of a larger domain $\0$  and either $\G(t) $ is a compact manifold without boundary 
\cite{ES98, ES96a, Mayer98, BGN10} or $\Gamma(t)$ intersects the boundary $\p\0$ of~$\0$ orthogonally~\cite{GR22, HRW21, ABCF00}.
Existence results in the setting of classical solutions have been established  almost simultaneously in \cite{CHF96, ES98, ES96a} under the assumption that $\Gamma(0)$ 
is a compact ${\rm C}^{k+\beta}$-hypersurface without boundary in $\R^n,$   $\beta\in(0,1)$
 and $n\geq 2$, with $k=3$ in \cite{CHF96} and $k=2$ in  \cite{ES98, ES96a}.
Subsequently, the well-posedness of the Mullins-Sekerka problem for $W^{1+3\mu-4/p}_p$ initial geometries, where $1/3 + (n + 3)/3p<\mu\leq1$, was proven in the recent monograph \cite{PS16}.
The existence theory in the situation with a contact angle condition of $\pi/2$ was established only recently in \cite{GR22, HRW21}.
We also refer to \cite{GR22, ES98, PS16} where stability issues are investigated and to \cite{BGN10, ZCH96, ET23} for numerical studies pertaining to this problem.
Finally, we mention  the papers  \cite{R05, BGS98, HS22xx, CL21, JMPS22x} where weak solutions to the Mullins-Sekerka problem are studied.

In this paper we consider the situation when the two phases are both unbounded and we restrict to the two-dimensional case.
To be more precise, we assume that at each time instant $t\geq0$ we have
\[
\0^\pm(t)=\{(x,y)\in\R^2\,:\, y\gtrless f(t,x)\}\qquad\text{and}\qquad \G(t):=\{(x,f(t,x)):\, x\in\R\},
\]
where $f(t):\R\to\R$, $t\geq0$, is an unknown function.
The same  setting has been also considered in \cite{COW19} where the authors establish convergence rates to a planar interface for global solutions (assuming they exist).
Our goal is to establish the well-posedness of the Mullins-Sekerka problem in this unbounded regime for initial data whose regularity is close of being optimal.
 To be more precise, the equations of motion are described by the following system  of equations
\begin{subequations}\label{PB}
\begin{equation}\label{PB1}
\left.
\arraycolsep=1.4pt
\begin{array}{rcllll}
\Delta u^\pm(t)&=&0&\quad\text{in $\Omega^\pm(t)$,}\\[1ex]
u^\pm(t)&=& \kappa_{\Gamma(t)}&\quad\text{on $\Gamma(t)$,}\\[1ex]
\nabla u^\pm(t) &\to&0&\quad\text{for $|(x,y)|\to\infty$,}\\[1ex]
V(t)&=&-[\p_{\nu_{\Gamma(t)}}  u(t)]&\quad\text{on $\Gamma(t)$}
\end{array}
\right\}
\end{equation}
for $t>0$.
Above, $\nu_{\Gamma(t)} $, $V(t)$, and~$\kappa_{\Gamma(t)}$  are the unit normal which points into $\Omega^+(t)$, the normal velocity, and the curvature of~$\Gamma(t)$.
Moreover, 
\[
[\p_{\nu_{\Gamma(t)}}  u(t)]:=\p_{\nu_{\Gamma(t)}}  u^+(t)-\p_{\nu_{\Gamma(t)}}  u^-(t),\qquad t>0,
\] 
represents the jump of $\nabla u(t)$ across $\Gamma(t)$ in normal direction.
The system \eqref{PB1} is supplemented by the initial condition
 \begin{equation}\label{IC}
 \arraycolsep=1.4pt
\begin{array}{rcllll}
f(0)&=&f_0.
\end{array}
\end{equation}
\end{subequations} 
Before presenting our main result, we emphasize that, under suitable conditions, the interface~$f(t)$ identifies at each 
time instant $t\ge 0$ the functions $u^\pm(t)$ uniquely, see Proposition~\ref{P:1} below.
Therefore, from now on, we shall  only refer    to $f$ as being a solution to \eqref{PB}.
A further observation is that if $f$ is a solution to \eqref{PB} then, given $\lambda>0$, also the function $f_\lambda$ with
\[
f_\lambda(t,x):=\lambda^{-1}f(\lambda^3 t, \lambda x),
\]
is a solution to \eqref{PB}. 
Since
\[
\Big\|\frac{d}{dx}f_\lambda(t)\Big\|_\infty=\Big\|\frac{d}{dx}f(\lambda^3 t)\Big\|_\infty\qquad\text{and}\qquad \|f_\lambda(t)\|_{{\dot{H}}^{3/2}}=\|f(\lambda^3 t)\|_{{\dot{H}}^{3/2}},
\]
where $\|\cdot\|_{{\dot{H}}^{3/2}}$ is the homogeneous Sobolev norm, we identify~${\rm BUC}^1(\R)$ and $H^{3/2}(\mathbb{R})$ as critical spaces for~\eqref{PB}.
In Theorem~\ref{MT1} we establish the well-posedness of~\eqref{PB} together with a parabolic smoothing property in all subcritical Sobolev spaces $H^r(\R)$ with $r\in(3/2,2)$.
With respect to this point we mention that all previous existence results  in the setting of classical solutions \cite{CHF96, ES98, ES96a, PS16,GR22, HRW21} consider  initial data
  with at least ${\rm C^2}$-regularity.

The main result of this paper is the following theorem.

\begin{thm}\label{MT1}
Let  $r\in(3/2,2)$ and choose $\ov r\in(3/2,r).$ Then,  given $f_0\in H^r(\R)$, there exists a unique maximal solution $f:=f(\,\cdot\,; f_0)$ to \eqref{PB} such that
\begin{equation*}  
\begin{aligned}
 &f\in {\rm C}([0,T^+),H^r(\mathbb{R}))\cap {\rm C}((0,T^+), H^{{\ov r}+1}(\mathbb{R}))\cap {\rm C}^1((0,T^+), H^{\ov r-2}(\mathbb{R})), \\
 &f(t)\in H^4(\R)\quad\text{for $t\in(0,T^+)$},\\
 &{u^\pm(t)\in{\rm C}^2(\0^\pm(t))\cap {\rm UC}^1(\0^\pm(t))}\quad\text{for $t\in(0,T^+)$},\\
 &\p_{\nu_{\G(t)}} u^\pm(t)\circ{\Xi_{\G(t)}}=(1+(f(t))'^2)^{-1/2}(\phi^\pm(t))' \quad\text{for $t\in(0,T^+)$ and some~${\phi^\pm(t) \in H^2(\R)},$}
\end{aligned}
  \end{equation*}
  where $T^+=T^+(f_0)\in(0,\infty]$ is the maximal existence time and $\Xi_{\G(t)}:\R\to \G(t)$ is defined by~${\Xi_{\G(t)}(x)=(x,f(t,x)).}$ 
Moreover,  $[(t,f_0)\mapsto f(t;f_0)]$ defines a semiflow on $H^r(\R)$ which is smooth in the open set
  \[
\{(t,f_0)\,:\, f_0\in H^r(\R),\, 0<t<T^+(f_0)\}\subset \R\times H^r(\R)  
  \]
  and
    \begin{equation}\label{eq:fg}
 f\in {\rm C}^\infty((0,T^+)\times\R,\R)\cap {\rm C}^\infty ((0, T^+),  H^k(\R))\quad \text{for all $k\in\N$}.
  \end{equation}
 \end{thm}

In Theorem~\ref{MT1} and below we let  $(\cdot)'$ denote the spatial derivative $d/dx$. 

The strategy to prove Theorem~\ref{MT1} consists in several steps. 
To begin with, we first prove that if $f(t)$ is known and belongs to $H^4(\R)$, then the first three equations of \eqref{PB1} identify the functions $u^\pm(t)$ uniquely, see Proposition~\ref{P:1}.
Furthermore, we can also represent the right side of \eqref{PB1}$_4$ in terms of certain singular integral operators which involve only the function $f(t,\cdot)$. 
In this way we reformulate the problem  as an evolution problem with only~$f$ as unknown, see~\eqref{NNEP}.
In the proof of Proposition~\ref{P:1} we rely on potential theory and some formulas, see Lemma~\ref{L:MP0-}~(iv), that relate the derivatives of certain  
singular integral operator evaluated at some density $\beta$ to the
$L_2$-adjoints of these operators evaluated  $\beta'$,  which have been used already in the context of the Muskat problem in~\cite{DGN23, MM23}.
Thanks to these formulas, we may formulate~\eqref{PB}, see \eqref{NNEP} in Section~\ref{Sec:3.1}, as an evolution problem in~$H^{r-2}(\R)$, $r\in(3/2,2)$, 
with nonlinearities which are expresses as a derivative.
Then, using a direct localization argument, we show in Section~\ref{Sec:3.2} that the problem is of parabolic type by identifying the right side of \eqref{NNEP} as the generator of an analytic semigroup.
The proof of the main result is established in Section~\ref{Sec:3.3} and relies on the quasilinear parabolic theory presented in~\cite{Am93, MW20}.

\subsection{Notation}\label{Sec:1.1} 
Given   Banach spaces $E_1$ and $E_0,$ we define $\kL(E_1,E_0)$ as the space of bounded linear operators from $E_1$ to $E_0$   and   $\kL(E_0):=\kL(E_0,E_0)$.
Moreover,  ${\rm Isom}(E_1,E_0) $ is the open subset of  $\kL(E_1,E_0)$ which consists  of isomorphisms and~${{\rm Isom}(E_0):= {\rm Isom}(E_0,E_0)}$. 
Furthermore $\kL^k_{\rm sym}(E_1,E_0),$ $k\geq1$, is the space of $k$-linear, bounded, and symmetric operators~${T:\;E_1^k\to E_0}$.
The set of all locally Lipschitz continuous mappings from~$E_1$ to~$E_0$ is denoted by ${\rm C}^{1-}(E_1,E_0)$  and  ${\rm C}^{\infty}(\cO,E_0)$ is the  
set which consists only of smooth mappings from an open set
 $\cO\subset E_1$ to $E_0$.
  
  If $E_1$ is additionally densely embedded in $E_0$, we   set (following~\cite{Am95})
\begin{equation*} 
\kH(E_1,E_0):=\{A\in\kL(E_1,E_0)\,:\, \text{$-A$ generates an analytic semigroup in $\kL(E_0)$}\}.
\end{equation*}  

Given a Banach space $E$, an interval $I\subset \R$, $n\in\N$, and~${\gamma\in(0,1)}$, we define ${\rm C}^{n}(I,E)$ as the set of all $n$-times continuously differentiable functions and 
${\rm C}^{n+\gamma}(I,E)$ is its  subset consisting of those functions which posses a locally $\gamma$-H\"older continuous $n$th derivative.
Moreover,  ${\rm BUC}^{n}(I,E)$  is the Banach space of functions with bounded and uniformly continuous derivatives up to order~$n$ and ${{\rm BUC}^{n+\gamma}(I,E)}$
 denotes its subspace   which consists  of those functions 
which have a  uniformly $\gamma$-H\"older continuous $n$th derivative. 
We also set~${{\rm BUC}^{\infty}(I,E)=\cap_{n\in\N} {\rm BUC}^{n}(I,E).}$
Finally, if $\0\subset\R^2$ is open and $n\in\N$, then~${\rm UC}^{n}(\0,E)$  is the set of functions with  uniformly continuous derivatives up to order~$n$.

 \section{Solvability of some boundary value problems}\label{Sec:2}

Our strategy to solve  \eqref{PB} is to reformulate   this model as  an evolution problem for the function~$f$  only.
To this end, we first solve via the method of potentials, for each given function~${f\in H^4(\R)},$ the (decoupled) boundary value problems for $u^+$ and $u^-$ given by the systems
\begin{equation}\label{PB1fixed}
\left.
\arraycolsep=1.4pt
\begin{array}{rcllll}
\Delta u^\pm&=&0&\quad\text{in $\Omega^\pm$,}\\[1ex]
u^\pm&=& \kappa_{\Gamma}&\quad\text{on $\Gamma$,}\\[1ex]
\nabla u^\pm &\to&0&\quad\text{for $|(x,y)|\to\infty$,}
\end{array}
\right\}
\end{equation}
where
\[
\0^\pm=\{(x,y)\in\R^2\,:\, y\gtrless f(x)\}\qquad\text{and}\qquad \G :=\p\0^\pm=\{(x, f(x)\,:\, x\in\R)\}.
\]
Below ${\nu_\G}$ is the outward unit normal at $\G$ which points into $\0^+$.
The corresponding  existence and uniqueness result is provided in Proposition~\ref{P:1} below.
Before stating this result we first introduce some notation.
Observe that~$\G$ is the image of the diffeomorphism $\Xi_\G:\R\to\G$ defined by~${\Xi(x):=(x,f(x))}$ for~$x\in\R$.
Then, the pulled-back curvature $\kappa(f):=\kappa_\G\circ\Xi_\G$ is given by the relation 
\begin{equation}\label{formulae}
\kappa(f):=\kappa_\G\circ\Xi_\G:=\Big(\frac{f'}{(1+f'^2)^{1/2}}\Big)'\quad\text{on $\R.$}
\end{equation}

Moreover, given functions $w^\pm\in {\rm C}(\overline{\0^\pm}),$ we set
 \begin{equation}\label{defjump}
 [w](x,f(x)):=w^+(x,f(x))-w^-(x,f(x)).
\end{equation}

\subsection{Some singular integral operators}\label{Sec:2.1}
We now introduce some singular integral operators which are used when solving~\eqref{PB1fixed}.
Given $f\in W^1_\infty(\R)$,  we set
\begin{equation}\label{OpAB}
 \begin{aligned}
 \bA(f)[\alpha](x)&:=\frac{1}{\pi}\PV\int_\R\frac{f'(x)-(\delta_{[x,s]}f)/s}{1+\big[(\delta_{[x,s]}f)/s\big]^2}\frac{\alpha(x-s)}{s}\, {\rm d}s,\\[1ex] 
 \bB(f)[\alpha](x)&:=\frac{1}{\pi}\PV\int_\R\frac{1+f'(x)(\delta_{[x,s]}f)/s}{1+\big[(\delta_{[x,s]}f)/s\big]^2}\frac{\alpha(x-s)}{s}\, {\rm d}s
 \end{aligned}
 \end{equation}
for $\alpha\in L_2(\R),$ where  $\PV$ is the  principal value  and  
 \[
 \delta_{[x,s]}f:=f(x)-f(x-s),\qquad x,\, s\in\R.
 \]
 Lemma~\ref{L:MP0}~(i) below ensures that these singular integral operators belong to  $\kL(L_2(\R))$.
 Their $L_2$-adjoints  are given by the relations
 \begin{equation}\label{adj}
\begin{aligned}
\bA(f)^*[\alpha](x)&=\frac{1}{\pi}\PV\int_\R\frac{\big(\delta_{[x,s]}f\big)/s-f'(x-s)}{1+\big[\big(\delta_{[x,s]}f\big)/s\big]^2}\frac{\alpha(x-s)}{s}\, {\rm d}s,\\[1ex]
\bB(f)^*[\alpha](x)&=-\frac{1}{\pi}\PV\int_\R\frac{1+f'(x-s)\big(\delta_{[x,s]}f\big)/s}{1+\big[\big(\delta_{[x,s]}f\big)/s\big]^2}\frac{\alpha(x-s)}{s}\, {\rm d}s.
\end{aligned}
\end{equation}
An important observation is that the operators defined in \eqref{OpAB}-\eqref{adj} can be represented in terms of a  family of singular integral operators $\{B_{n,m}^0(f)\,:\, n,\, m\in\N\}$ which we now introduce.
Given~${n,\,m\in\N}$ and  Lipschitz continuous  mappings~${a_1,\ldots, a_{m},\, b_1, \ldots, b_n:\mathbb{R}\to\mathbb{R}}$,  we set
\begin{equation}\label{BNM}
B_{n,m}(a_1,\ldots, a_m)[b_1,\ldots,b_n,\alpha](x):=
\frac{1}{\pi}\PV\int_\mathbb{R}  \cfrac{\prod_{i=1}^{n}\big(\delta_{[x,s]} b_i\big) /y}{\prod_{i=1}^{m}\big[1+\big[\big(\delta_{[x,s]}  a_i\big) /s\big]^2\big]}\frac{\alpha(x-s)}{s}\, {\rm d}s 
\end{equation}
for $\alpha\in L_2(\R)$.
In particular, if    ${f:\mathbb{R}\to\mathbb{R}}$  is Lipschitz continuous  we use the short notation
\begin{equation}\label{defB0}
B^0_{n,m}(f) :=B_{n,m}(f,\ldots, f)[f,\ldots,f,\cdot].
\end{equation} 
These operators have been defined in the context of the Muskat problem in \cite{MBV18}.
It is now a straight forward consequence of \eqref{OpAB}-\eqref{defB0} to observe that 
\begin{equation}\label{adjBNM}
\begin{aligned}
&\bA(f)[\alpha] =f'B_{0,1}^0(f)[\alpha]-B_{1,1}^0(f)[\alpha],\qquad&&\bA(f)^*[\alpha] =B_{1,1}^0(f)[\alpha]-B_{0,1}^0(f)[f'\alpha],\\[1ex]
&\bB(f)[\alpha] =B_{0,1}^0(f)[\alpha]+f'B_{1,1}^0(f)[\alpha],\qquad &&\bB(f)^*[\alpha] =-B_{0,1}^0(f)[\alpha]-B_{1,1}^0(f)[f'\alpha].
\end{aligned}
\end{equation}
In view of the representation \eqref{adjBNM} several mapping properties for the operators introduced in  \eqref{OpAB}-\eqref{adj} can be derived from the following result.

\begin{lemma}\label{L:MP0} Let  $n,\,m \in\N$. 
\begin{itemize}
\item[(i)] Let $a_1,\ldots, a_m:\R\to\R$ be Lipschitz continuous mappings.
Then, there exists a positive constant~$C=C(n,\, m,\,\max_{i=1,\ldots, m}\|a_i'\|_{\infty} )$
such that for all  Lipschitz continuous functions~${b_1,\ldots, b_n:\R\to\R}$ we have
\[
\|B_{n,m}(a_1,\ldots, a_m)[b_1,\ldots,b_n,\,\cdot\,]\|_{\kL(L_2(\R))}\leq C\prod_{i=1}^{n} \|b_i'\|_{\infty}.
\]
 Moreover, $B_{n,m}\in {\rm C}^{1-}(W^1_\infty(\R)^{m},\kL^n_{\rm sym}(W^1_\infty(\R),\kL(L_2(\R)))).$
 \item[(ii)] Given $k\geq 2$, it holds that  $B_{n,m} \in {\rm C}^{1-}(H^k(\R)^m,{\kL}^n_{\rm sym}(H^k(\R),\kL(H^{k-1}(\R)))).$
  \item[(iii)]  Given $r\in (3/2,2)$, it holds that $[f\mapsto B_{n,m}^0(f)]\in {\rm C}^{\infty}(H^r(\R),{\kL}(H^{r-1}(\R))).$
  \item[(iv)]  Let $r\in (3/2,2)$ and $a_1,\ldots, a_m\in H^r(\R)$ be given. Then,  there exists  a  positive constant~$C=C(n,\, m,\,\max_{i=1,\ldots, m}\|a_i \|_{H^r} )$  such that 
  for all~${b_1,\ldots, b_n\in H^{r}(\R)}$  we have 
  \[
\|B_{n,m}(a_1,\ldots, a_m)[b_1,\ldots,b_n,\,\cdot\,]\|_{\kL(H^{r-2}(\R))}\leq C\prod_{i=1}^{n} \|b_i \|_{H^r}.
\]
\end{itemize} 
\end{lemma}
\begin{proof}
The property  (i) is established in \cite[Lemma 3.1]{MBV18}. 
The claim (ii) is proven for $k=2$ in   \cite[Lemma 4.3]{MP2022} and the case $k\geq 3$ follows from this result via induction.
Moreover, (iii) is established in \cite[Appendix C]{MP2021} and (iv) in \cite[Lemma 2.5]{MM21}.
\end{proof}

The next lemma collects some  important properties of the operators defined in \eqref{OpAB}-\eqref{adj}.

\begin{lemma}\label{L:MP0-} Let $\lambda\in\R\setminus(-1,1).$ 
\begin{itemize}
\item[(i)] If $f\in{\rm BUC}^1(\R)$,  then $\lambda-\bA(f),\, \lambda-\bA(f)^*\in{\rm Isom}(L_2(\R))$. 
\item[(ii)] If $f\in H^r(\R)$, $r\in(3/2,2)$,  then $\lambda-\bA(f),\, \lambda-\bA(f)^*\in{\rm Isom}(H^{r-1}(\R))$.  
\item[(iii)] If $f\in H^2(\R)$, then $\lambda-\bA(f),\, \lambda-\bA(f)^*\in{\rm Isom}(H^{1}(\R))$.  
\item[(iv)] If $f\in H^2(\R)$ and $\beta\in H^1(\R)$, then $\bA(f)^*[\beta]$ and $\bB(f)^*[\beta]$  belong to $H^1(\R)$ with 
\[
(\bA(f)^*[\beta])'=-\bA(f)[\beta']\qquad\text{and}\qquad (\bB(f)^*[\beta])'=-\bB(f)[\beta'].
\]
\item[(v)] If $f\in H^3(\R)$, then $ \lambda-\bA(f)^*\in{\rm Isom}(H^{2}(\R))$.  
\end{itemize}
\end{lemma}
\begin{proof}
The property (i) follows from \cite[Theorem~3.5]{MBV18}
and (ii) is established in \cite[Theorem~5]{AM22} and \cite[Propositin 3.4]{MM23}.
Moreover, the claim~(iii) is proven  in \cite[Proposition 3.6 and Lemma 3.8]{MBV18} and~(iv)  in \cite[Proposition 2.3]{MM23}.
The assertion  (v) is a consequence of (iii) and~(iv).
Indeed, given $f\in H^3(\R)$, $\lambda\in\R\setminus(-1,1),$ and~${\alpha\in H^2(\R)}$,  the properties (iii)-(iv) imply that  $(\lambda-\bA(f)^*)[\alpha] \in H^1(\R)$ with  
\[
((\lambda-\bA(f)^*)[\alpha])'=(\lambda+\bA(f))[\alpha'] \in H^1(\R).
\]
Hence, $(\lambda-\bA(f)^*)[\alpha] \in H^2(\R)$ and
\begin{align*}
2\|(\lambda-\bA(f)^*)[\alpha]\|_{H^2}^2&\geq\|(\lambda-\bA(f)^*)[\alpha]\|_{H^1}^2+\|((\lambda-\bA(f)^*)[\alpha])'\|_{H^1}^2\\[1ex]
&=\|(\lambda-\bA(f)^*)[\alpha]\|_{H^1}^2+\|(\lambda+\bA(f))[\alpha']\|_{H^1}^2\\[1ex]
&\geq C(\|\alpha\|^2_{H^1}+\|\alpha'\|^2_{H^1})\\[1ex]
&\geq C\|\alpha\|_{H^2}^2,
\end{align*}
the inequalities in the second last line of the formula (with a sufficiently small constant $C$ independent of $\lambda$ and $\alpha$) being a straightforward consequence of (iii).
The assertion (v) follows now from this estimate via the method of continuity \cite[Proposition I.1.1.1]{Am95}.
\end{proof}

\subsection{The solvability of the boundary value problems~\eqref{PB1fixed}}\label{Sec:2.2}

As a preliminary result we provide in Proposition~\ref{P:A1} the unique solvability of a transmission type boundary value problem which is used to establish the uniqueness claim in Proposition~\ref{P:1} below.

\begin{prop}\label{P:A1} Given $f\in H^3(\R)$ and $\phi\in H^2(\R)$, the boundary value problem
\begin{equation}\label{BVP}
\left.
\arraycolsep=1.4pt
\begin{array}{rcllll}
\Delta U^\pm&=&0&\quad\text{in $\Omega^\pm$,}\\[1ex]
[U]&=& 0&\quad\text{on $\Gamma$,}\\[1ex]
[\p_{\nu_\G} U]&=&\big((1+f'^2)^{-1/2}\phi'\big)\circ\Xi_\G^{-1}&\quad\text{on $\Gamma$,}\\[1ex]
\nabla U^\pm &\to&0&\quad\text{for $|(x,y)|\to\infty$,}\\[1ex]
\end{array}
\right\}
\end{equation}
 has a solution $(U^+,U^-)$ such that~${U^\pm\in{\rm C}^2(\0^\pm)\cap {\rm UC}^1(\0^\pm)}$.
 Moreover, the solution is, up to an additive constant, unique.
\end{prop}
\begin{proof} We first prove uniqueness of solutions in the class described above.
Let therefore~$U$ be a solution to the homogeneous problem associated with \eqref{BVP} (that is with $\phi=0$).
Setting~${U:=U_+{\bf 1}_{\0^+}+ U_-{\bf 1}_{\0^-},}$ Stokes' theorem leads us to  the equation
  \[
\Delta U=0\quad\text{in\quad $\mathcal{D}'(\R^2).$}  
  \]
  Hence, $U$ is the real part of a holomorphic function $h:\mathbb{C}\to\mathbb{C}$.
Since $h'$ is also holomorphic and $h'=\nabla U$ is bounded and vanishes for $|(x,y)|\to\infty$, it follows that $h'=0$, meaning that~$U$ is constant in $\R^2$. 

In order to establish the existence of solutions, we set 
\begin{equation}\label{rvar}
r=(r_1,r_2)=(x-s,y-f(s)) \quad\text{for $s\in\R$ and $(x,y)\in\R^2\setminus\G.$}
\end{equation}
Defining $U:\R^2\setminus\G\to\R$ by the formula
\begin{equation}\label{u}
U(x,y):= \frac{1}{2\pi}\int_\R\frac{ r_1+f'(s)r_2}{|r|^2}\phi(s)\, {\rm d}s
\end{equation}
and setting $U^\pm:=U|_{\0^\pm}$, we next show that $(U^+,U^-)$ is a solution to \eqref{BVP} with the required properties.
To start, we note that
\[
U(x,y)=\int_\R K(x,y,s)\phi(s)\,{\rm d}s, \qquad\text{$(x,y)\in\R^2\setminus\G,$}
\]
and, for every $\alpha\in\N^2,$ we have $\partial^\alpha_{(x,y)} K(x,y,s)=O(s^{-1})$ for $|s|\to\infty$ and locally uniformly in~${(x,y)\in \R^2\setminus\Gamma}$.
This shows that $U$ is well-defined and that integration and differentiation with respect to $x$ and $y$ may be commuted.

Furthermore, we infer from \eqref{u}   and \cite[Lemma~A.1]{BM22}, that $U^\pm\in{\rm C}^\infty(\0^\pm)\cap {\rm UC}(\0^\pm) $ 
with~${[U]=0},$ the gradient $\nabla U=(\p_x U,\p_y U)$ being  given by the formula
\begin{equation}\label{nablau}
\nabla U(x,y)=\frac{1}{2\pi}\int_\R\frac{1}{|r|^4}\raisebox{1.1ex}{$(-f'(s)\;\;1)$}
  \begin{pmatrix}
    2r_1r_2 &r_2^2-r_1^2 \\[1ex]
   r_2^2-r_1^2 & -2r_1r_2
    \end{pmatrix}\phi(s)\, {\rm d}s, \qquad (x,y)\in\R^2\setminus\G.
\end{equation}
Using the  matrix identity
\[
\frac{1}{|r|^4}\raisebox{1.1ex}{$(-f'(s)\;\;1)$}
  \begin{pmatrix}
    2r_1r_2 &r_2^2-r_1^2 \\[1ex]
   r_2^2-r_1^2 & -2r_1r_2
    \end{pmatrix}=-\p_s\frac{(r_1,r_2)}{|r|^2}
\]
 together with integration by parts we obtain that 
\begin{equation}\label{nablau'}
\nabla U(x,y)=\frac{1}{2\pi}\int_\R\frac{(r_1,r_2)}{|r|^2}\phi'(s)\, {\rm d}s, \qquad (x,y)\in\R^2\setminus\G.
\end{equation}
Combining \eqref{nablau'} and  \cite[Lemma~A.1, Lemma A~4]{BM22}, we  obtain that $U^\pm\in{\rm UC}^1(\0^\pm)$ satisfies also \eqref{BVP}$_4$ and
\[
[\p_{\nu_\G} U] = \big((1+f'^2)^{-1/2}\phi'\big)\circ\Xi_\G^{-1}
\]
It is now easy to infer from \eqref{nablau'} that also \eqref{BVP}$_1$ holds true, and therewith we established the existence of a solution.\medskip
 
\end{proof}

We are now in a position to solve the boundary value problems \eqref{PB1fixed} for $u^+$ and $u^-$.
\begin{prop}\label{P:1}
Given $f\in H^4(\R)$, there exist unique solutions~${u^\pm\in{\rm C}^2(\0^\pm)\cap {\rm UC}^1(\0^\pm)}$ to~\eqref{PB1fixed}  such  that
 $\p_{\nu_\G} u^\pm\circ\Xi_\G=(1+f'^2)^{-1/2}(\phi^\pm)'$ for some functions~${\phi^\pm \in H^2(\R)}$.
Furthermore, $\nabla u^\pm=v^\pm$ in $\0^\pm$, where
\begin{equation}\label{velocities}
v^\pm(x,y):=\frac{1}{2\pi}\int_\R\frac{(f(s)-y,x-s)}{(x-s)^2+(y-f(s))^2}\alpha^\pm(s)\, {\rm d}s,\qquad (x,y)\in\0^\pm,
\end{equation}
and with density functions $\alpha^\pm\in H^1(\R) $ given by the relation
 \begin{equation}\label{densfunc}
\alpha^\pm=2(\mp1+\bA(f))^{-1}[(\kappa(f))']\in H^1(\R).
\end{equation}
\end{prop}
\begin{proof}
\noindent{\em (i) Existence.}
According to  Lemma~\ref{L:MP0-}~(iii) we have $\mp 1+\bA(f)\in{\rm Isom}(H^1(\R))$ and, since~$(\kappa(f))'\in H^1(\R),$ the density functions $\alpha^\pm$ 
defined in \eqref{densfunc} are well-defined and belong to~${H^1(\R)}$.
We next infer from \cite[Lemma A.1]{BM22} that  the vector fields $v^\pm$ defined in~\eqref{velocities} belong to~${{\rm C}^\infty(\0^\pm)\cap {\rm UC}(\0^\pm)}$ and
\begin{align}\label{Vary}
v^\pm\circ\Xi_\G(x)=\frac{1}{2\pi }\PV\int_\R\frac{(f(s)-f(x),x-s)}{(x-s)^2+(f(x)-f(s))^2}\alpha^\pm(s)\, {\rm d}s\mp\frac{1}{2}\frac{\alpha^\pm(1,f')}{1+f'^2}(x),\quad x\in\R.
\end{align} 
Moreover, $v^\pm$ satisfies the asymptotic boundary condition $v^\pm(x,y)\to 0$ for $|(x,y)|\to\infty$ and
\begin{equation*} 
{\rm div\, } v^\pm={\rm rot\, } v^\pm=0\qquad\text{in $\0^\pm$,}
\end{equation*}
 see \cite[Lemma A.4]{BM22}.
 Setting $v^\pm=(v^\pm_1,v^\pm_2)$, the relation ${\rm rot\, } v^\pm=0$ in $\0^\pm$ ensures that the functions 
 \[
u^\pm(x,y):=c^\pm +\int_0^xv^\pm_1(s,\pm d)\, ds+\int_{\pm  d}^yv^\pm_2(x,s)\, {\rm d}s,  \qquad (x,y)\in\0^\pm,
 \] 
 where $c^\pm\in\R$ and $d>\|f\|_\infty$, satisfy $\nabla u^\pm=v^\pm$ in~$\0^\pm$.
Moreover, $u^\pm\in{{\rm C}^2(\0^\pm)\cap {\rm UC}^1(\0^\pm)}$ and, since $v^\pm$ are divergence free, \eqref{PB1fixed}$_1$ is satisfied.
It is clear that also the asymptotic  boundary conditions \eqref{PB1fixed}$_2$ hold.
Combining \eqref{OpAB}, \eqref{Vary}, and the  relation~${\nabla u^\pm=v^\pm}$ on~$\G$,  we further have
\begin{equation}\label{norder}
\p_{\nu_\G} u^\pm\circ\Xi_\G=\frac{(1+f'^2)^{-1/2}}{2}\bB(f)[\alpha^\pm].
\end{equation}
In order to show that $\bB(f)[\alpha^\pm]$ are derivatives of functions in $H^2(\R)$ we define  $\beta^\pm\in H^2(\R)$ by the relations
\begin{equation}\label{densfunc2}
\beta^\pm=2(\mp1-\bA(f)^*)^{-1}[\kappa(f)],
\end{equation}
see Lemma~\ref{L:MP0-}~(v).
We next differentiate \eqref{densfunc2} with respect to $x$ and infer then from  Lemma~\ref{L:MP0-}~(iii)-(iv) that $(\beta^\pm)'=\alpha^\pm$ and 
\[
\bB(f)[\alpha^\pm]=\bB(f)[(\beta^\pm)']=-(\bB(f)^*[\beta^\pm])'.
\]
Setting $\phi^\pm :=-\bB(f)^*[\beta^\pm/2]$, it follows from \eqref{adjBNM} and Lemma~\ref{L:MP0}~(ii) that $\phi^\pm \in H^2(\R)$.
 Moreover,  \eqref{norder} lead to $\p_{\nu_\G} u^\pm\circ\Xi_\G=(1+f'^2)^{-1/2}(\phi^\pm)'$.
As a final step we show that the additive constants~$c^\pm$ can be chosen such that also \eqref{PB1fixed}$_2$ are satisfied.
 Indeed,  in view of~\eqref{densfunc} and~\eqref{Vary}, we have
 \begin{align*}
\frac{d}{dx}(u^\pm|_\G\circ\Xi_\G)= (1,f')\cdot v^\pm|_\Gamma\circ\Xi_\G =\frac{1}{2}(\mp 1+\bA(f))[\alpha^\pm]=(\kappa(f))',
 \end{align*}
so that    $u^\pm|_\G\circ\Xi_\G-\kappa(f)$   is a constant function. 
Therewith, we established the existence of a solution to \eqref{PB1fixed}.\medskip

\noindent{\em (ii) Uniqueness.} 
It suffices to show that the homogeneous  problems
\begin{equation}\label{PB10}
\left.
\arraycolsep=1.4pt
\begin{array}{rcllll}
\Delta u^\pm&=&0&\quad\text{in $\Omega^\pm$,}\\[1ex]
u^\pm&=& 0&\quad\text{on $\Gamma$,}\\[1ex]
\nabla u^\pm &\to&0&\quad\text{for $|(x,y)|\to\infty$,}
\end{array}
\right\}
\end{equation}
have unique solutions $u^\pm$  with the required properties.
We establish only the uniqueness of~$u^+$ (that of $u^-$ follows by similar arguments).
Let thus $\phi^+\in H^2(\R)$ be the function  which satisfies the relation  $\p_{\nu_\G} u^+\circ\Xi_\G=(1+f'^2)^{-1/2}(\phi^+)'$.
Setting $U^-:=0$ and $U^+:=u^+$, we note that $(U^+,U^-)$ solves the boundary value problem~\eqref{BVP} (with $\phi=\phi^+$) and it is thus given  by the formula \eqref{u}.
In particular, it follows from \eqref{u} and \cite[Lemma A.1]{BM22} that
\[
0=U^-|_\G\circ\Xi_\G=U^+|_\G\circ\Xi_\G=-\frac{1}{2}\bB(f)^*[\phi'],
\]
and together with Lemma~\ref{L:MP0-}~(iv) we get
\[
0=-(\bB(f)^*[\phi'])'=\bB(f)[\phi'']. 
\] 
However, as shown in \cite[Eqs. (3.22) and (3.25)]{MBV18}, there exits a positive  constant~$C $ such that~${\|\bB(f)[\alpha]\|_2\geq C\|\alpha\|_2}$ for all $\alpha\in L_2(\R)$. 
Therefore $\phi''=0$, hence also $\phi=0$.
We now infer from \eqref{u}  that $U^+=u^+=0$, and the uniqueness claim is proven. 
\end{proof}

 \section{The evolution problem and the proof of the main result}\label{Sec:3}

In this section we first formulate the original problem \eqref{PB} as an evolution problem for~$f$, see~\eqref{NNEP}.
Subsequently, we prove that the linearization of the right side of \eqref{NNEP} is the generator of an analytic semigroup, see Theorem~~\ref{T:GP} below, and we conclude the section with the proof 
of the main result stated in Theorem~\ref{MT1}.
\subsection{The evolution problem}\label{Sec:3.1}
In order to formulate the system  \eqref{PB} as an evolution problem for $f$ we first infer from Proposition~\ref{P:1} that if $(f,u^\pm)$ is a solution to \eqref{PB}  as stated in Theorem~\ref{MT1}, 
then, for each $t>0$, we have
\begin{equation*}
\p_{\nu_{\G(t)}} u^\pm(t)\circ\Xi_{\G(t)}=-(1+f'^2(t))^{-1/2}(\bB(f(t))^*[(\mp1-\bA(f(t))^*)^{-1}[\kappa(f(t))]])'.
\end{equation*}
Together with \eqref{PB}$_4$ we arrive at the following evolution equation
\[
\frac{df}{dt}(t)=(\bB(f(t))^*[((-1-\bA(f(t))^*)^{-1}-(1-\bA(f(t))^*)^{-1})[\kappa(f(t))]])'\qquad\text{for  $t>0$}.
\]
As we want to solve the latter equation in the phase space $H^r(\R)$ with $r\in(3/2,2)$ we encounter the problem that
 the curvature $\kappa(f)$ is in general not a function, but a distribution.  
However, taking full advantage of the quasilinear character of the curvature operator 
we can formulate  the system~\eqref{PB} as the following quasilinear evolution problem
\begin{equation}\label{NNEP}
\frac{df}{dt}(t)=\Phi(f(t))[f(t)],\quad t>0,\qquad f(0)=f_0,
\end{equation}
where $\Phi:H^r(\R)\to\kL(H^{r+1}(\R), H^{r-2}(\R))$ is defined by the following formula
\begin{equation}\label{phi}
\Phi(f)[h]:=(\bB(f)^*[((-1-\bA(f)^*)^{-1}-(1-\bA(f)^*)^{-1})[\kappa(f)[h]]])',
\end{equation}
with $\kappa:H^r(\R)\to\kL(H^{r+1}(\R),H^{r-1}(\R))$ given by
\begin{equation}\label{curv}
\kappa(f)[h]:=\frac{h''}{(1+f'^2)^{3/2}}.
\end{equation}
We point out that, if $f\in H^2(\R)$, then $\kappa(f)[f]$ is exactly the pulled-back curvature $\kappa(f)$.
Moreover, arguing as in \cite[Appendix~C]{MP2021}, it is not difficult to prove that 
\begin{equation}\label{reg:curve}
\kappa\in {\rm C}^\infty(H^r(\R),\kL(H^{r+1}(\R), H^{r-1}(\R))).
\end{equation}
Recalling~\eqref{adjBNM}, it follows from  Lemma~\ref{L:MP0}~(iii) and Lemma~\ref{L:MP0-}~(ii), by also using the  smoothness of the map which associate to an isomorphism its inverse, that 
\begin{equation}\label{reg:B*A}
\bB(f)^*,\, (\pm1-\bA(f)^*)^{-1}\in {\rm C}^\infty(H^r(\R),\kL(H^{r-1}(\R))).
\end{equation}
Gathering \eqref{phi}-\eqref{reg:B*A}, we obtain in view of ${d/dx\in\kL(H^{r-1}(\R), H^{r-2}(\R))}$ that 
\begin{equation}\label{reg:PHI}
\Phi\in {\rm C}^\infty(H^r(\R),\kL(H^{r+1}(\R), H^{r-2}(\R))).
\end{equation}

\subsection{The parabolicity property}\label{Sec:3.2}
Our next goal is to prove that the problem~\eqref{NNEP} is of parabolic  type  in the sense that, for each $f\in H^r(\R)$, $r\in(3/2,2)$, the operator $\Phi(f)$ 
is the generator of an analytic semigroup in~$\kL(H^{r-2}(\R))$. 
This is the content of the next result.

\begin{thm}\label{T:GP}
Given  $f\in H^r(\R)$, $r\in(3/2,2),$ it holds that $-\Phi(f)\in\mathcal{H}(H^{r+1}(\R), H^{r-2}(\R)).$
\end{thm}
In the proof of Theorem~\ref{T:GP} we exploit of the fact that, given $h\in H^{r+1}(\R)$, the action~${\Phi(f)[h]}$ is the derivative of a function  which lies in $H^{r-1}(\R)$.
The proof of Theorem~\ref{T:GP} is postponed to the end of this subsection and it relies on a strategy inspired  by \cite{E94,  ES95,  ES97}. 

As a first step we associate to $\Phi(f)$ the continuous path
\[
[\tau\mapsto \Phi(\tau f)]:[0,1]\to \mathcal{L}(H^{r+1}(\R), H^{r-2}(\R)),
\]
and we note that 
\[
\Phi(0)=-2\frac{d}{dx}B(0)^*\frac{d^2}{dx^2}= 2H\frac{d^3}{dx^3},
\]
where $H$ is the Hilbert transform. In particular, $\Phi(0)$ is the Fourier multiplier defined by the  symbol $[\xi\mapsto2|\xi|^3].$
 As a second step we locally approximate in  Proposition~\ref{P:AP} the operator~$\Phi(\tau f)$ by  Fourier multipliers which coincide, up to some positive multiplicative constants, with~$\Phi(0)$. 
As a final third  step we  establish for these Fourier multipliers  suitable (uniform) resolvent estimates, see \eqref{L:FM1}-\eqref{L:FM2}. 
The proof of Theorem~\ref{T:GP} follows then by combining the results established in these three steps.

Before presenting  Proposition~\ref{P:AP}, we  choose  for each~${\varepsilon\in(0,1)}$, a finite $\varepsilon$-localization family, that is  a family  
\[\{(\pi_j^\varepsilon, x_j^\e)\,:\, -N+1\leq j\leq N\}\subset  {\rm C}^\infty(\mathbb{R},[0,1])\times\R,\]
with $N=N(\varepsilon)\in\mathbb{N} $ sufficiently large, such that $x_j^\e\in\supp\pi_j^\e$, $-N+1\leq j\leq N$, and
\begin{align*}
\bullet\,\,\,\, \,\,  &\mbox{$\supp \pi_j^\varepsilon \subset\{|x|\leq \varepsilon+1/\varepsilon\}$ is an interval of length $\varepsilon$ for $\ |j|\leq N-1$;}\\
 \bullet\,\,\,\, \,\, &\mbox{$\supp \pi_{N}^\varepsilon\subset\{|x|\geq 1/\e\}$;} \\[1ex]
\bullet\,\,\,\, \,\, &\mbox{ $ \pi_j^\varepsilon\cdot  \pi_l^\varepsilon=0$ if $[|j-l|\geq2, \max\{|j|, |l|\}\leq N-1]$ or $[|l|\leq N-2, j=N];$} \\[1ex]
\bullet\,\,\,\, \,\, &\mbox{ $\displaystyle\sum_{j=-N+1}^N(\pi_j^\varepsilon)^2=1;$} \\[1ex]
 \bullet\,\,\,\, \,\, &\mbox{$\|(\pi_j^\varepsilon)^{(k)}\|_\infty\leq C\varepsilon^{-k}$ for all $ k\in\mathbb{N}, -N+1\leq j\leq N$.} 
\end{align*} 
To each finite $\varepsilon$-localization family we associate  a second family   
$$\{\chi_j^\varepsilon\,:\, -N+1\leq j\leq N\}\subset {\rm C}^\infty(\mathbb{R},[0,1])$$ with the following properties
\begin{align*}
\bullet\,\,\,\, \,\,  &\mbox{$\chi_j^\varepsilon=1$ on $\supp \pi_j^\varepsilon$ for $-N+1\leq j\leq N$ and $\supp\chi_N^\varepsilon\subset \{|x|\geq 1/\varepsilon-\varepsilon\}$;} \\[1ex]
\bullet\,\,\,\, \,\,  &\mbox{$\supp \chi_j^\varepsilon$ is an interval  of length $3\varepsilon$ and with the same midpoint as $ \supp \pi_j^\varepsilon$, $|j|\leq N-1$.} 
\end{align*} 
It is not difficult to prove that, given $r\in\R$ and $\e\in(0,1),$  there exists $c=c(\e,r)\in(0,1)$ such that for all $h\in H^r(\R)$ we have
\begin{equation}\label{equivnorm}
c\|h\|_{H^r}\leq \sum_{j=-N+1}^N\|\pi_j^\varepsilon h\|_{H^r}\leq c^{-1}\|h\|_{H^r}. 
\end{equation}

We are now in a position  to establish the aforementioned localization result.
\begin{prop}\label{P:AP} 
Let  $3/2<r'<r<2,$   $f\in H^r(\R)$, and~$\nu>0$ be given. 
Then, there exist~${\e\in(0,1)}$, a $\e$-locali\-za\-tion family  $\{(\pi_j^\e,x_j^\e)\,:\, -N+1\leq j\leq N\} $, and  a constant~${K=K(\e)}$
 such that 
 \begin{equation}\label{D1}
  \|\pi_j^\e \Phi(\tau f) [h]-2a_{\tau_,j}\Phi(0)[\pi_j^\e h]\|_{H^{r-2}}\leq \nu \|\pi_j^\e h\|_{H^{r+1}}+K\|  h\|_{H^{ r'+1}}
 \end{equation}
 for all $ -N+1\leq j\leq N$, $\tau\in[0,1],$  and  $h\in H^{r+1}(\R)$,  where, letting $a_\tau:=(1+\tau^2f'^2)^{-3/2}$, we set
\[ 
a_{\tau,N}:=\lim_{|x|\to\infty}a_{\tau}(x)=1\qquad\text{and}\qquad a_{\tau,j}:= a_\tau(x_j^\e),\quad |j|\leq N-1.
\]
\end{prop}
\begin{proof} 
In the following    $C$ and $C_0$ are constants that do not depend on $\e$, while constants denoted by $K$ may depend on $\e$.

Given $-N+1\leq j\leq N$, $\tau\in[0,1],$  and~$h\in H^{r+1}(\R)$ we  have
\begin{align*}
&\|\pi_j^\e \Phi(\tau f) [h]-a_{\tau_,j}\Phi(0)[\pi_j^\e h]\|_{H^{r-2}}\\[1ex]
&=\|\pi_j^\e \Phi(\tau f) [h] -a_{\tau_,j}(H[(\pi_j^\e h)''])'\|_{H^{r-2}}\\[1ex]
&\leq \|(\pi_j^\e \bB(\tau f)^*[((-1-\bA(\tau f)^*)^{-1} -(1-\bA(\tau f)^*)^{-1})[\kappa(\tau f)[h]]])'-a_{\tau_,j}(H[(\pi_j^\e h)''])'\|_{H^{r-2}}\\[1ex]
&\qquad+\|(\pi_j^\e)' \bB(\tau f)^*[((-1-\bA(\tau f)^*)^{-1}-(1-\bA(f)^*)^{-1})[\kappa(\tau f)[h]]]\|_{H^{r-2}},
\end{align*}
where, in view of \eqref{adjBNM}, \eqref{curv}, Lemma~\ref{L:MP0}~(i), and Lemma~\ref{L:MP0-}~(i) we have
\begin{align*}
&\|(\pi_j^\e)' \bB(\tau f)^*[((-1-\bA(\tau f)^*)^{-1}-(1-\bA(f)^*)^{-1})[\kappa(\tau f)[h]]]\|_{H^{r-2}}\\[1ex]
&\leq K\|\bB(\tau f)^*[((-1-\bA(\tau f)^*)^{-1}-(1-\bA(f)^*)^{-1})[\kappa(\tau f)[h]]]\|_{H^{r-2}}\\[1ex]
&\leq   K\|\bB(\tau f)^*[((-1-\bA(\tau f)^*)^{-1}-(1-\bA(f)^*)^{-1})[\kappa(\tau f)[h]]]\|_{2}\\[1ex]
&\leq   K\|\kappa(\tau f)[h]\|_{2}\\[1ex]
&\leq K\|h\|_{H^{ r'+1}}.
\end{align*}
Since  ${d/dx\in\kL(H^{r-1}(\R), H^{r-2}(\R))}$ is a contraction, we have shown that 
\begin{equation}\label{EST:0}
\begin{aligned}
&\|\pi_j^\e \Phi(\tau f) [h]-2a_{\tau_,j}\Phi(0)[\pi_j^\e h]\|_{H^{r-2}}\\[1ex]
&\leq \|\pi_j^\e \bB(\tau f)^*[((-1-\bA(\tau f)^*)^{-1} -(1-\bA(\tau f)^*)^{-1})[\kappa(\tau f)[h]]]-2a_{\tau_,j}H[(\pi_j^\e h)'']\|_{H^{r-1}}\\[1ex]
&\qquad+K\|h\|_{H^{ r'+1}}.
\end{aligned}
\end{equation}
In  remains to estimated the first term on the right of \eqref{EST:0}. To this end several steps are needed.\medskip

\noindent{\em Step 1.} 
Given $\tau\in[0,1]$ and $h\in H^{1+r}(\R)$ we define $\vartheta^\pm(\tau)[h]\in H^{r-1}(\R)$ as the unique solutions to
\begin{equation}\label{thetatau}
(\pm1-\bA(\tau f)^*)[\vartheta^\pm(\tau)[h]]=\kappa(\tau f)[h],
\end{equation}
see \eqref{reg:curve} and Lemma~\ref{L:MP0}~(ii).
In this step we  prove there exists a  constant $C_0>0$ such that  for all $\e\in(0,1)$, $\tau\in[0,1]$, $-N+1\leq j\leq N$, and $h\in H^{r+1}(\R)$ we have
 \begin{equation}\label{x-1}
 \|\vartheta^\pm(\tau)[h]\|_{H^{r-1}}\leq C_0\|\pi_j^\e h\|_{H^{r+1}}+ K\|h\|_{H^{r'+1}}.
 \end{equation}
 Indeed, after multiplying  \eqref{thetatau} by~$\pi_j^\e,$ we arrive at
 \begin{align*} 
(\pm1-\bA(\tau f)^*)[\pi_j^\e \vartheta^\pm(\tau )[h]]&= \pi_j^\e \kappa(\tau f)[h]-(\bA(\tau f)^*[\pi_j^\e \vartheta^\pm(\tau )[h]]-\pi_j^\e\bA(\tau f)^*[\vartheta^\pm(\tau )[h]]),
\end{align*}
and it can bee easily shown that 
\[
\|\pi_j^\e\kappa(\tau f)[h]\|_{H^{r-1}}\leq C\|\pi_j^\e h\|_{H^{r+1}}+ K\|h\|_{H^{r'+1}}.
\]
Moreover,  since $r-1<1$, the commutator estimate in Lemma~\ref{L:AL1} together with \eqref{adjBNM}  yields 
\[
\|\bA(\tau f)^*[\pi_j^\e \vartheta^\pm(\tau )[h]]-\pi_j^\e\bA(\tau f)^*[ \vartheta(\tau )[h]]\|_{H^{r-1}}\leq K\|\vartheta(\tau )[h]\|_2 \leq  K\|  h\|_{H^{ r'+1}}.
\]
 The estimates \eqref{x-1} follow now from Lemma~\ref{L:MP0-}~(ii).\medskip
 
 \noindent{\em Step 2.}  
Recalling~\eqref{adjBNM}, we infer from  Lemma~\ref{L:AL2} if $|j|\leq N-1$, respectively from  Lemma~\ref{L:AL3} if~$j=N$, that, if $\e\in(0,1)$ is sufficiently small,
 then for all $\tau\in[0,1]$, $-N+1\leq j\leq N$, and~${h\in  H^{r+1}(\R)}$ we have  
 \begin{align*}
 \|\pi_j^\e \bB(\tau f)^*[\vartheta^\pm(\tau )[h]]+H[\pi^\e_j  \vartheta^\pm(\tau )[h]]\|_{H^{r-1}}\leq \frac{\nu}{4C_0}\|\pi^\e_j\vartheta^\pm(\tau )[h]\|_{H^{r-1}}+K\|\vartheta^\pm(\tau )[h]\|_{_{H^{r'-1}}}.
\end{align*}  
The estimates \eqref{x-1} and the property \eqref{reg:curve} (with $r=r'$) enable us to conclude that for all~${\tau\in[0,1]}$, $-N+1\leq j\leq N$, and $h\in  H^{r+1}(\R)$ it holds  that
 \begin{equation*}
 \|\pi_j^\e \bB(\tau f)^*[\vartheta^\pm(\tau )[h]]+H[\pi^\e_j  \vartheta^\pm(\tau )[h]]\|_{H^{r-1}}\leq \frac{\nu}{4}\|\pi^\e_j h\|_{H^{r+1}}+K\|h\|_{_{H^{r'+1}}},
\end{equation*}  
provided that $\e$ is  sufficiently small, and therefore
\begin{equation}\label{x-2}
\begin{aligned}
 &\|\pi_j^\e \bB(\tau f)^*[\vartheta^-(\tau )[h]-\vartheta^+(\tau )[h]]+H[\pi^\e_j (\vartheta^-(\tau )[h]-\vartheta^+(\tau )[h])]\|_{H^{r-1}}\\[1ex]
&\leq \frac{\nu}{2}\|\pi^\e_j h\|_{H^{r+1}}+K\|h\|_{_{H^{r'+1}}}.
\end{aligned}
\end{equation}  
\medskip

  \noindent{\em Step 3.} 
We show that, if~${\e\in(0,1)}$ is  sufficiently small, then for all~$\tau\in[0,1]$, $-N+1\leq j\leq N$, and~$h\in  H^{r+1}(\R)$ we have
 \begin{equation}\label{x-3}
\|\pm H[\pi^\e_j  \vartheta^\pm(\tau )[h]]- a_{\tau_,j}H[(\pi_j^\e h)'']\|_{H^{r-1}}\leq \frac{\nu}{4 }\|\pi^\e_j h\|_{H^{r+1}}+K\|h\|_{_{H^{r'+1}}}. 
\end{equation}  
To start with, we note that since   $H\in\kL(H^{r-1}(\R))$ is  an isometry we have
\[ 
 \|\pm H[\pi^\e_j  \vartheta^\pm(\tau )[h]]-a_{\tau_,j}H[(\pi_j^\e h)'']\|_{H^{r-1}}\leq\|\pm\pi^\e_j  \vartheta^\pm(\tau )[h]- a_{\tau_,j}(\pi_j^\e h)''\|_{H^{r-1}},
 \] 
and it remains to estimate the right side of the latter inequality. 
To this end we first infer from~\eqref{thetatau} that  
\begin{equation*}
\pm\pi^\e_j  \vartheta^\pm(\tau )[h]- a_{\tau_,j}(\pi_j^\e h)''=\pi^\e_j\bA(\tau f)^*[\vartheta^\pm (\tau)[h]]
+\pi^\e_j\kappa(\tau f)[h]- a_{\tau_,j}(\pi_j^\e h)''.
\end{equation*} 
Noticing that $\|a_\tau\|_{{\rm BUC}^{s-3/2}}\leq3\|f'\|_{{\rm BUC}^{s-3/2}} $ and~${{a_\tau(x)\to1}}$ for~$|x|\to\infty$ uniformly with respect to $\tau\in[0,1]$
and using the  estimate
\begin{align*}
\|g_1g_2\|_{H^{r-1}}\leq  C(\|g_1\|_\infty\|g_2\|_{H^{r-1}}+\|g_2\|_\infty\|g_1\|_{H^{r-1}})\qquad\text{for $g_1,\, g_2\in H^{r-1}(\mathbb{R}),$}
\end{align*} 
we have in view of $\chi_j^\e\pi_j^\e=\pi_j^\e$ that
\begin{align*}
\|\pi^\e_j\kappa(\tau f)[h]- a_{\tau_,j} (\pi_j^\e h)'' \|_{H^{r-1}}&\leq \|(a_\tau-a_\tau(x_j^\e))(\pi_j^\e h)''\|_{H^{r-1}}+K\|h\|_{H^{r'+1}} \\[1ex]
 &\leq C\|\chi_j^\e(a_\tau-a_\tau(x_j^\e))\|_\infty\|\pi_j^\e h\|_{H^{r+1}}+K\|h\|_{H^{r'+1}} \\[1ex]
 &\leq \frac{\nu}{8} \|\pi_j^\e h\|_{H^{r+1}}+K\|h\|_{H^{r'+1}} 
 \end{align*}
  for all $\tau\in[0,1]$, $|j|\leq N-1$, and $h\in  H^{r+1}(\R)$, provided that $\e$ is sufficiently small.
Similarly, for $j=N$ we  have
  \begin{align*}
\|\pi^\e_N\kappa(\tau f)[h]-a_{\tau_,N}(\pi_N^\e h)''\|_{H^{r-1}} &\leq \|(a_\tau-1)(\pi_N^\e h)''\|_{H^{r-1}}+K\|h\|_{H^{r'+1}} \\[1ex]
 &\leq C\|\chi_N^\e(a_\tau-1)\|_\infty\|\pi_N^\e h\|_{H^{r+1}}+K\|h\|_{H^{r'+1}} \\[1ex]
 &\leq \frac{\nu}{8} \|\pi_N^\e h\|_{H^{r+1}}+K\|h\|_{H^{r'+1}}. 
 \end{align*}
Furthermore,  appealing to Lemma~\ref{L:AL2} if $|j|\leq N-1$, respectively to Lemma~\ref{L:AL3} if $ j=N$, we find together with the representation \eqref{adjBNM} of $\bA(\tau f)^*$ 
that, if $\e$ is sufficiently small, then
 \[
 \|\pi^\e_j\bA(\tau f)^*[\vartheta^\pm (\tau)[h]]\|_{H^{r-1}}\leq \frac{\nu}{8 C_0}\|\pi^\e_j\vartheta^\pm(\tau )[h]\|_{H^{r-1}}+K\|\vartheta^\pm(\tau )[h]\|_{_{H^{r'-1}}},
 \]
 and, together with  \eqref{x-1} and the property \eqref{reg:curve} (with $r=r'$),   we get
  \[
 \|\pi^\e_j\bA(\tau f)^*[\vartheta^\pm (\tau)[h]]\|_{H^{r-1}}\leq \frac{\nu}{8 }\|\pi^\e_jh\|_{H^{r+1}}+K\|h\|_{_{H^{r'+1}}}
 \]
 for all $\tau\in[0,1]$, $-N+1\leq j\leq N$, and $h\in  H^{r+1}(\R)$.
This proves \eqref{x-3}.\medskip

 Combining  the estimates \eqref{EST:0}, \eqref{x-2}, and \eqref{x-3}, we  conclude that \eqref{D1} holds true and this completes  the proof.
\end{proof}

In Proposition~\ref{P:AP} we have locally approximated $\Phi(\tau f)$ by Fourier multipliers $2a_{\tau,j} \Phi(0)$, and, since $f'$ is a bounded function,
 there exists a constant $\eta=\eta(\|f'\|_\infty)\in(0,1)$ such that~${2a_{\tau,j}\in[\eta,\eta^{-1}]}$.
Elementary Fourier analysis arguments enable us to conclude there exists a constant $\kappa_0=\kappa_0(\eta)\geq1$ such  that for all $\delta\in[\eta,\eta^{-1}]$ and all $\re\lambda\geq 1$  we have  
 \begin{align}
\bullet &\,\, \text{$\lambda-\delta\Phi(0)\in {\rm Isom}(H^{r+1}(\R),H^{r-2}(\R))$},\label{L:FM1}\\[1ex]
\bullet &\,\,  \kappa_0\|(\lambda-\delta\Phi(0))[h]\|_{H^{r-2}}\geq |\lambda|\cdot\|h\|_{H^{r-2}}+\|h\|_{H^{r+1}} \quad  \text{for all  $h\in H^{r+1}(\R)$.}\label{L:FM2}
\end{align}

 We are now in a position to establish Theorem~\ref{T:GP}.
 \begin{proof}[Proof of Theorem~\ref{T:GP}]
Let $\kappa_0\geq1$ be as identified in \eqref{L:FM2}. 
Setting $\nu:=(2\kappa_0)^{-1}$, Proposition~\ref{P:AP} ensures there exist~${\e\in(0,1)}$, a $\e$-locali\-za\-tion family $\{(\pi_j^\e,x_j^\e)\,:\, -N+1\leq j\leq N\},$ and  a constant~${K=K(\e)}$ such
 that for all $\tau\in[0,1],$  $ -N+1\leq j\leq N$,  and  $h\in H^{r+1}(\R)$ we have
 \begin{equation*} 
  \|\pi_j^\e \Phi(\tau f) [h]-2a_{\tau_,j}\Phi(0)[\pi_j^\e h]\|_{H^{r-2}}\leq \nu \|\pi_j^\e h\|_{H^{r+1}}+K\|  h\|_{H^{ r'+1}}.
 \end{equation*} 
Recalling~\eqref{L:FM2}, we also have 
  \begin{equation*} 
    \kappa_0\|(\lambda-2a_{\tau_,j}\Phi(0))[\pi^\e_jh]\|_{H^{r-2}}\geq |\lambda|\cdot\|\pi^\e_jh\|_{H^{r-2}}+ \|\pi^\e_j h\|_{H^{r+1}}
 \end{equation*}
 for all  $\tau\in[0,1],$ $-N+1\leq j\leq N$,  $\re \lambda\geq 1$, and $h\in H^{r+1}(\R)$.
Combining these estimates, we get
 \begin{align*}
   2\kappa_0\|\pi_j^\e(\lambda-\Phi(\tau f))[h]\|_{H^{r-2}}\geq& 2\kappa_0\|(\lambda-2a_{\tau_,j}\Phi(0))[\pi^\e_j h]\|_{H^{r-2}}\\[1ex]
   &\qquad-2\kappa_0\|\pi_j^\e\Phi(\tau f)[h]-2a_{\tau_,j}\Phi(0)[\pi^\e_j h]\|_{H^{r-2}}\\[1ex]
   \geq& 2|\lambda|\cdot\|\pi^\e_j h\|_{H^{r-2}}+ \|\pi^\e_j h\|_{H^{r+1}}-2\kappa_0K\|  h\|_{H^{r'+1}}.
 \end{align*}
 Summing  up over $j$, the estimates \eqref{equivnorm},  Young's inequality, and the interpolation property
\begin{align}\label{IP}
[H^{s_0}(\R),H^{s_1}(\R)]_\theta=H^{(1-\theta)s_0+\theta s_1}(\R),\qquad\theta\in(0,1),\, -\infty< s_0\leq s_1<\infty,
\end{align}
cf. e.g. \cite[Section 2.4.2/Remark~2]{Tr78},  where $[\cdot,\cdot]_\theta$ is the complex interpolation functor, imply there exist constants~$\kappa\geq1$  and~$\omega\geq1 $ such that 
for all   $\tau\in[0,1],$   $\re \lambda\geq \omega$, and  $h\in H^{r+1}(\R)$ we have
  \begin{align}\label{KDED}
   \kappa\|(\lambda-\Phi(\tau f ))[h]\|_{H^{r-2}}\geq |\lambda|\cdot\|h\|_{H^{r-2}}+ \| h\|_{H^{r+1}}.
 \end{align}
The property \eqref{L:FM1} together with the method of continuity \cite[Proposition I.1.1.1]{Am95} and \eqref{KDED} now yield
\begin{align}\label{DEDK2}
   \omega-\Phi(f)\in {\rm Isom}(H^{r+1}(\R), H^{r-2}(\R)).
 \end{align}
The desired generator property follows now directly from   \eqref{KDED} (with $\tau=1$) and \eqref{DEDK2}, see \cite[Chapter I]{Am95}.
\end{proof}

\subsection{The proof of the main result}\label{Sec:3.3}
We complete this section with the proof of the main result which exploits the abstract quasilinear parabolic theory presented in \cite{Am93} (see also \cite[Theorem 1.1]{MW20}).

\begin{proof}[Proof of Theorem~\ref{MT1}]
Let $E_1:=H^{\ov r+1}(\R)$, $E_0:=H^{\ov r-2}(\R)$, and~${E_\theta:=[E_0,E_1]_\theta},\, {\theta\in(0,1)}$.
Defining $\beta:=2/3$ and $\alpha:=(r-\ov r+2)/3$, it holds that $0<\beta<\alpha<1$,  ${E_\beta=H^{\ov r}(\R)}$, and~${E_\alpha=H^{r}(\R)}$.
Theorem~\ref{T:GP} together with the regularity property \eqref{reg:PHI} (both with~${r=\ov r}$) ensure that $-\Phi \in {\rm C}^{\infty}(E_\beta,\kH(E_1, E_0))$.
This enables us to apply  \cite[Theorem 1.1]{MW20}  in the context of the quasilinear parabolic evolution  problem~\eqref{NNEP}.
Consequently, given~${f_0\in H^r(\R)}$, there exists a unique maximal classical solution  $f= f(\,\cdot\, ; f_0)$ to~\eqref{NNEP} such that
 \begin{equation}\label{classol} 
 f\in {\rm C}([0,T^+),H^r(\mathbb{R}))\cap {\rm C}((0,T^+), H^{{\ov r}+1}(\mathbb{R}))\cap {\rm C}^1((0,T^+), H^{{\ov r}-2}(\mathbb{R})) 
  \end{equation}
   and  
  \begin{equation}\label{lochoel}
f\in   {\rm C}^{\zeta}([0,T^+), H^{\ov r}(\mathbb{R}),
  \end{equation}
  where $T^+=T^+(f_0)\in(0,\infty]$ is the maximal existence time and  $\zeta\in(0,\alpha-\beta]$  can be chosen arbitrary small, cf. \cite[Remark 1.2~(ii)]{MW20}.
  Moreover, the mapping $[(t,f_0)\mapsto f(t;f_0)]$ defines a semiflow on $H^r(\R)$ which is smooth in the open set
  \[
\{(t,f_0)\,:\, f_0\in H^r(\R),\, 0<t<T^+(f_0)\}\subset \R\times H^r(\R).  
  \]
  
  We next prove that the uniqueness claim holds in the class of classical solutions; that is of solutions which satisfy merely~\eqref{classol}.
  To this end prove that each such solution with the property~\eqref{classol} satisfies~\eqref{lochoel} for some small $\zeta.$
  Let therefore $T\in(0,T^+)$ be arbitrary but fixed.
  Then, there exists a positive constant~$C$ such that for all $t\in[0,T]$ we have
  \begin{equation}\label{est0}
  \|\kappa(f(t))\|_{H^{r-2}}=\Big\|\Big(\frac{(f(t))'}{(1+(f(t))')^{1/2}}\Big)'\Big\|_{H^{r-2}}\leq\Big\| \frac{(f(t))'}{(1+(f(t))'^2)^{1/2}} \Big\|_{H^{r-1}}\leq C.
  \end{equation}
  Moreover, in virtue of Lemma~\ref{L:MP0-}~(i)-(ii) $\pm 1-\bA(f(t))\in {\rm Isom}(L_2(\R))\cap {\rm Isom}(H^{r-1}(\R))$ for all $t\in[0,T]$,  and together with \eqref{IP} and the observation that 
  $0<2-r<r-1$ we get that $\pm 1-\bA(f(t))\in  {\rm Isom}(H^{2-r}(\R))$.
  Since by Lemma~\ref{L:MP0}~(i)-(ii) and \eqref{adjBNM} the mapping 
  \[
  [t\mapsto \bA(f(t))]:[0,T]\to \kL(L_2(\R))\cap \kL(H^{r-1}(\R))
  \]
   is in particular continuous, we may chose $C>0$ sufficiently large to guarantee that for all~${t\in[0,T]}$ it holds that 
 \begin{equation}\label{est1}
  \|(\pm1-\bA(f(t))^{-1}\|_{\kL(H^{2-r}(\R))} \leq C.
  \end{equation}
  Therefore, setting $\vartheta^\pm(t):=(\pm1-\bA(f(t)^*)^{-1}[\kappa(f(t))]\in H^{\ov r-1}(\R),$ $t\in(0,T]$, we  infer from~\eqref{est0} and \eqref{est1} that
   there exists a constant $C>0$ such that for all $t\in(0,T]$ we have
  \begin{align*}
  \|\vartheta^\pm(t)\|_{H^{r-2}}&=\sup_{\|\psi\|_{ H^{2-r}}=1}|\langle \vartheta^\pm(t)| \psi\rangle_2| =\sup_{\|\psi\|_{ H^{2-r}}=1}|\langle (\pm1-\bA(f(t)^*)^{-1}[\kappa(f(t))]| \psi\rangle_2|\\[1ex]
  &=\sup_{\|\psi\|_{ H^{2-r}}=1}|\langle \kappa(f(t))| (\pm1-\bA(f(t))^{-1}[\psi]\rangle_2|\\[1ex]
  &\leq \sup_{\|\psi\|_{ H^{2-r}}=1}\|\kappa(f(t))\|_{H^{r-2}}\|(\pm1-\bA(f(t))^{-1}[\psi]\|_{ H^{2-r}}\\[1ex]
  &\leq C.
  \end{align*}
  Above $\langle\cdot|\cdot\rangle_2$ is the $L_2$-scalar product.
Since $\Phi(f(t))[f(t)]=(\bB(f(t))^*[\vartheta^-(t)-\vartheta^+(t)])'$ for~${t\in(0,T]}$, see~\eqref{phi},  it follows now from Lemma~\ref{L:MP0}~(iv) and \eqref{adjBNM} there exists a constant~${C>0}$ such that 
 for all $t\in(0,T]$  
 \begin{align*}
 \|\Phi(f(t))[f(t)]\|_{H^{r-3}}&\leq \|\bB(f(t))^*[\vartheta^-(t)-\vartheta^+(t)]\|_{H^{r-2}}\leq C(1+\|(f(t))'\vartheta^\pm(t)\|_{H^{r-2}}\leq C.
 \end{align*}
 To derive  the last inequality we have use the continuity of the multiplication operator
 \[
[(g_1,g_2)\mapsto g_1g_2]:H^{r-1}(\R)\times H^{2-r}(\R)\to H^{2-r}(\R), 
 \]
 see \cite[Eq. (1.8)]{MM21}.
 To summarize, we have shown that 
 \[
 \sup_{t\in(0,T]}\Big\|\frac{df}{dt}(t)\Big\|_{H^{r-3}}\leq C.
 \]
 Since $f\in {\rm C}([0,T],H^r(\mathbb{R}))$, the latter estimate together with the mean value theorem and the observation that for $\zeta:=(r-\ov r)/3 $ it holds that
 $[H^{r-3}(\R), H^r(\R)]_{1-\zeta}=H^{\ov r}(\R)$, see~\eqref{IP}, yields
 \[
\|f(t_2)-f(t_1)\|_{H^{\ov r}}\leq C\|f(t_2)-f(t_1)\|_{H^{\ov r-3}}^{\zeta} \leq C |t_2-t_1|^{\zeta} \qquad\text{for all $0\leq t_1\leq t_2\leq T$,}
\]
which proves \eqref{lochoel}.
Recalling  Proposition~\ref{P:1}, we have established the existence and uniqueness of   maximal classical solutions to \eqref{PB}.
Finally, the parabolic smoothing property~\eqref{eq:fg}  may be shown by using a parameter trick  employed also  in other settings, see \cite{An90, ES96, PSS15, MBV19}. 
Since the arguments are more or less identical to those used in \cite[Theorem 1.3]{MBV19}, we refrain to present them here.

\end{proof}

\appendix
\section{Some properties  of the singular integral operators $B_{n,m}^0(f)$}\label{Appendix A}

We recall some recent results that are available for the singular integrals operators~$B_{n,m}^0(f)$ introduced in~\eqref{defB0} and which are used in the analysis in Section~\ref{Sec:3}.
We begin with a commutator type estimate.
\begin{lemma}\label{L:AL1} 
Let $n,\, m \in \N$,   $r\in(3/2,2)$, $f\in H^r(\R)$, and  ${\varphi\in {\rm BUC}^1(\R)}$ be given. 
Then, there exists  a constant $K$ that  depends only on~$ n,$~$m, $~$\|\varphi'\|_\infty, $ and~$\|f\|_{H^r}$  such that for all~${\alpha\in L_2(\R)}$  we have
 \begin{equation*} 
  \|\varphi B_{n,m}^0(f)[\alpha]- B_{n,m}^0(f)[ \varphi \alpha]\|_{H^1}\leq K\| \vartheta\|_{2}
 \end{equation*}
\end{lemma}
\begin{proof}
See \cite[Lemma 12]{AM22}.
\end{proof}

The next results describe how  to localize the singular integrals operators~$B_{n,m}^0(f)$.  
They may be viewed as generalizations of the method of freezing the coefficients of elliptic differential operators.

\begin{lemma}\label{L:AL2} 
Let $n,\, m \in \N$, $r\in(3/2,2),$ $r'\in(3/2,r)$, and  $\nu\in(0,1)$ be given. 
Let further~${f\in H^r(\mathbb{R})}$ and   $a \in \{1\}\cup H^{r-1}(\mathbb{R})$.
For any sufficiently small $\e\in(0,1)$, there exists
a constant $K$ that depends only on $\e,\, n,\, m,\, \|f\|_{H^r},$  and  $\|a\|_{H^{r-1}}$ (if $a\neq1$)   such that for all~${|j|\leq N-1}$ and  $\alpha\in H^{r-1}(\mathbb{R})$ we have
 \begin{equation*} 
  \Big\|\pi_j^\e  B_{n,m}^0(f)[ a\alpha]-\frac{a(x_j^\e) (f'(x_j^\e))^n}{[1+(f'(x_j^\e))^2]^m}H[\pi_j^\e \alpha]\Big\|_{H^{r-1}}\leq \nu \|\pi_j^\e  \alpha\|_{H^{r-1}}+K\| \alpha\|_{H^{r'-1}}. 
 \end{equation*}
\end{lemma}  
\begin{proof}
See  \cite[Lemma~13]{AM22} if $a=1$, respectively   \cite[Lemma~D.5]{MP2021} if~${a\in H^{r-1}(\R)}$.
\end{proof}

 Lemma~\ref{L:AL3}  below describes how to localize the operators $B^0_{n,m}(f)$ ``at infinity''.

\begin{lemma}\label{L:AL3} 
Let $n,\, m \in \N$, $r\in(3/2,2),$ $r'\in(3/2,r)$, and  $\nu\in(0,1)$ be given. 
Let further~${f\in H^r(\mathbb{R})}$    and~${a\in\{1\}\cup H^{r-1}(\mathbb{R})}$.
For any sufficiently small  $\e\in(0,1)$, there exists a constant~$K$ that depends only on~$\e,\, n,\, m,\, \|f\|_{H^r},$   and $\|a\|_{H^{r-1} }$ (if $a\neq1$)  such that for all~${\alpha\in H^{r-1}(\mathbb{R})}$ 
  \begin{equation*}
  \|\pi_N^\e  B_{n,m}^0(f)[a \alpha]\|_{H^{r-1}}\leq \nu \|\pi_N^\e \alpha\|_{H^{r-1}}+K\| \alpha\|_{H^{r'-1}}\qquad \text{if $n\geq 1$ or $a\in H^{r-1}(\R)$},
 \end{equation*} 
  and
  \begin{equation*}
  \|\pi_N^\e  B_{0,m}^0(f)[\alpha]-H[\pi_N^\e \alpha]\|_{H^{r-1}}\leq \nu \|\pi_N^\e \alpha\|_{H^{r-1}}+K\| \alpha\|_{H^{r'-1}}
 \end{equation*} 
\end{lemma}  
\begin{proof}
See \cite[Lemma~15]{AM22} if $a=1$, respectively  \cite[Lemma D.6]{MP2021} if $a\in H^{r-1}(\R)$.
\end{proof}

\bibliographystyle{siam}
\bibliography{L_MS}

\begin{thebibliography}{10}

\bibitem{AM22}
{\sc H.~Abels and B.-V. Matioc}, {\em {Well-posedness of the Muskat problem in
  subcritical $L_p$-Sobolev spaces}}, European J. Appl. Math., 33 (2022),
  pp.~224--266.

\bibitem{HRW21}
{\sc H.~Abels, M.~Rauchecker, and M.~Wilke}, {\em Well-posedness and
  qualitative behaviour of the {M}ullins-{S}ekerka problem with ninety-degree
  angle boundary contact}, Math. Ann., 381 (2021), pp.~363--403.

\bibitem{ABC94}
{\sc N.~D. Alikakos, P.~W. Bates, and X.~Chen}, {\em Convergence of the
  {C}ahn-{H}illiard equation to the {H}ele-{S}haw model}, Arch. Rational Mech.
  Anal., 128 (1994), pp.~165--205.

\bibitem{ABCF00}
{\sc N.~D. Alikakos, P.~W. Bates, X.~Chen, and G.~Fusco}, {\em
  Mullins-{S}ekerka motion of small droplets on a fixed boundary}, J. Geom.
  Anal., 10 (2000), pp.~575--596.

\bibitem{Am93}
{\sc H.~Amann}, {\em {Nonhomogeneous linear and quasilinear elliptic and
  parabolic boundary value problems}}, in {Function spaces, differential
  operators and nonlinear analysis ({F}riedrichroda, 1992)}, vol.~133 of
  {Teubner-Texte Math.}, Teubner, Stuttgart, 1993, pp.~9--126.

\bibitem{Am95}
\leavevmode\vrule height 2pt depth -1.6pt width 23pt, {\em {Linear and
  Quasilinear Parabolic Problems. {V}ol. {I}}}, vol.~89 of {Monographs in
  Mathematics}, Birkh\"auser Boston, Inc., Boston, MA, 1995.
\newblock Abstract linear theory.

\bibitem{An90}
{\sc S.~B. Angenent}, {\em {Nonlinear analytic semiflows}}, Proc. Roy. Soc.
  Edinburgh Sect. A, 115 (1990), pp.~91--107.

\bibitem{BGN10}
{\sc J.~W. Barrett, H.~Garcke, and R.~N\"{u}rnberg}, {\em On stable parametric
  finite element methods for the {S}tefan problem and the {M}ullins-{S}ekerka
  problem with applications to dendritic growth}, J. Comput. Phys., 229 (2010),
  pp.~6270--6299.

\bibitem{BM22}
{\sc J.~Bierler and B.-V. Matioc}, {\em {The multiphase Muskat problem with
  equal viscosities in two dimensions}}, Interfaces Free Bound., 24 (2022),
  pp.~163--196.

\bibitem{BGS98}
{\sc L.~Bronsard, H.~Garcke, and B.~Stoth}, {\em A multi-phase
  {M}ullins-{S}ekerka system: matched asymptotic expansions and an implicit
  time discretisation for the geometric evolution problem}, Proc. Roy. Soc.
  Edinburgh Sect. A, 128 (1998), pp.~481--506.

\bibitem{CL21}
{\sc A.~Chambolle and T.~Laux}, {\em Mullins-{S}ekerka as the {W}asserstein
  flow of the perimeter}, Proc. Amer. Math. Soc., 149 (2021), pp.~2943--2956.

\bibitem{CHF96}
{\sc X.~Chen, J.~Hong, and F.~Yi}, {\em Existence, uniqueness, and regularity
  of classical solutions of the {M}ullins-{S}ekerka problem}, Comm. Partial
  Differential Equations, 21 (1996), pp.~1705--1727.

\bibitem{COW19}
{\sc O.~Chugreeva, F.~Otto, and M.~G. Westdickenberg}, {\em Relaxation to a
  planar interface in the {M}ullins-{S}ekerka problem}, Interfaces Free Bound.,
  21 (2019), pp.~21--40.

\bibitem{DGN23}
{\sc H.~Dong, F.~Gancedo, and H.~Q. Nguyen}, {\em Global well-posedness for the
  one-phase {M}uskat problem}, Comm. Pure Appl. Math., n/a.

\bibitem{EGK17}
{\sc C.~Eck, H.~Garcke, and P.~Knabner}, {\em Mathematical modeling}, Springer
  Undergraduate Mathematics Series, Springer, Cham, 2017.

\bibitem{E94}
{\sc J.~Escher}, {\em {The {D}irichlet-{N}eumann operator on continuous
  functions}}, Ann. Scuola Norm. Sup. Pisa Cl. Sci. (4), 21 (1994),
  pp.~235--266.

\bibitem{ES95}
{\sc J.~Escher and G.~Simonett}, {\em {Maximal regularity for a free boundary
  problem}}, 2 (1995), pp.~463--510.

\bibitem{ES96}
\leavevmode\vrule height 2pt depth -1.6pt width 23pt, {\em {Analyticity of the
  interface in a free boundary problem}}, Math. Ann., 305 (1996), pp.~439--459.

\bibitem{ES96a}
\leavevmode\vrule height 2pt depth -1.6pt width 23pt, {\em On {H}ele-{S}haw
  models with surface tension}, Math. Res. Lett., 3 (1996), pp.~467--474.

\bibitem{ES97}
\leavevmode\vrule height 2pt depth -1.6pt width 23pt, {\em {Classical solutions
  of multidimensional Hele-Shaw models}}, SIAM J. Math. Anal., 28 (1997),
  pp.~1028--1047.

\bibitem{ES98}
\leavevmode\vrule height 2pt depth -1.6pt width 23pt, {\em A center manifold
  analysis for the {M}ullins-{S}ekerka model}, J. Differential Equations, 143
  (1998), pp.~267--292.

\bibitem{ET23}
{\sc T.~Eto}, {\em A rapid numerical method for the mullins-sekerka flow with
  application to contact angle problems},  (2023).
\newblock arXiv:2202.13261.

\bibitem{Ga13}
{\sc H.~Garcke}, {\em Curvature driven interface evolution}, Jahresber. Dtsch.
  Math.-Ver., 115 (2013), pp.~63--100.

\bibitem{GR22}
{\sc H.~Garcke and M.~Rauchecker}, {\em Stability analysis for stationary
  solutions of the {M}ullins-{S}ekerka flow with boundary contact}, Math.
  Nachr., 295 (2022), pp.~683--705.

\bibitem{Gu93}
{\sc M.~E. Gurtin}, {\em Thermomechanics of evolving phase boundaries in the
  plane}, Oxford Mathematical Monographs, The Clarendon Press, Oxford
  University Press, New York, 1993.

\bibitem{HS22xx}
{\sc S.~Hensel and K.~Stinson}, {\em {Weak solutions of Mullins-Sekerka flow as
  a Hilbert space gradient flow}}.
\newblock arXiv::2206.08246, 2022.

\bibitem{JMPS22x}
{\sc V.~Julin, M.~Morini, M.~Ponsiglione, and E.~Spadaro}, {\em {The
  asymptotics of the area-preserving mean curvature and the {Mullins–Sekerka}
  flow in two dimensions}}, Math. Ann.,  (2022).
\newblock arXiv:2112.13936.

\bibitem{MM21}
{\sc A.-V. Matioc and B.-V. Matioc}, {\em {The Muskat problem with surface
  tension and equal viscosities in subcritical $L_p$-Sobolev spaces}}, J.
  Elliptic Parabol. Equ., 7 (2021), pp.~635--670.

\bibitem{MM23}
\leavevmode\vrule height 2pt depth -1.6pt width 23pt, {\em A new reformulation
  of the {M}uskat problem with surface tension}, J. Differential Equations, 350
  (2023), pp.~308--335.

\bibitem{MBV18}
{\sc B.-V. Matioc}, {\em {Viscous displacement in porous media: the Muskat
  problem in 2D}}, Trans. Amer. Math. Soc., 370 (2018), pp.~7511--7556.

\bibitem{MBV19}
\leavevmode\vrule height 2pt depth -1.6pt width 23pt, {\em {The Muskat problem
  in two dimensions: equivalence of formulations, well-posedness, and
  regularity results}}, Anal. PDE, 12 (2019), pp.~281--332.

\bibitem{MP2021}
{\sc B.-V. Matioc and G.~Prokert}, {\em Two-phase {S}tokes flow by capillarity
  in full 2d space: an approach via hydrodynamic potentials}, Proc. Roy. Soc.
  Edinburgh Sect. A, 151 (2021), pp.~1815--1845.

\bibitem{MP2022}
\leavevmode\vrule height 2pt depth -1.6pt width 23pt, {\em {Two-phase Stokes
  flow by capillarity in the plane: The case of different viscosities}}, NoDEA
  Nonlinear Differential Equations Appl., 29 (2022), pp.~Paper No. 54, 34.

\bibitem{MW20}
{\sc B.-V. Matioc and C.~Walker}, {\em On the principle of linearized stability
  in interpolation spaces for quasilinear evolution equations}, Monatsh. Math.,
  191 (2020), pp.~615--634.

\bibitem{Mayer98}
{\sc U.~F. Mayer}, {\em Two-sided {M}ullins-{S}ekerka flow does not preserve
  convexity}, in Proceedings of the {T}hird {M}ississippi {S}tate {C}onference
  on {D}ifference {E}quations and {C}omputational {S}imulations ({M}ississippi
  {S}tate, {MS}, 1997), vol.~1 of Electron. J. Differ. Equ. Conf., Southwest
  Texas State Univ., San Marcos, TX, 1998, pp.~171--179.

\bibitem{ON01}
{\sc B.~Niethammer and F.~Otto}, {\em Ostwald ripening: the screening length
  revisited}, Calc. Var. Partial Differential Equations, 13 (2001), pp.~33--68.

\bibitem{PSS15}
{\sc J.~Pr\"uss, Y.~Shao, and G.~Simonett}, {\em {On the regularity of the
  interface of a thermodynamically consistent two-phase {S}tefan problem with
  surface tension}}, Interfaces Free Bound., 17 (2015), pp.~555--600.

\bibitem{PS16}
{\sc J.~Pr\"uss and G.~Simonett}, {\em Moving Interfaces and Quasilinear
  Parabolic Evolution Equations}, vol.~105 of Monographs in Mathematics,
  Birkh\"auser/Springer, [Cham], 2016.

\bibitem{R05}
{\sc M.~R\"{o}ger}, {\em Existence of weak solutions for the
  {M}ullins-{S}ekerka flow}, SIAM J. Math. Anal., 37 (2005), pp.~291--301.

\bibitem{S96}
{\sc B.~E.~E. Stoth}, {\em Convergence of the {C}ahn-{H}illiard equation to the
  {M}ullins-{S}ekerka problem in spherical symmetry}, J. Differential
  Equations, 125 (1996), pp.~154--183.

\bibitem{Tr78}
{\sc H.~Triebel}, {\em {Interpolation Theory, Function Spaces, Differential
  Operators}}, North-Holland, Amsterdam, 1978.

\bibitem{ZCH96}
{\sc J.~Zhu, X.~Chen, and T.~Y. Hou}, {\em An efficient boundary integral
  method for the {M}ullins-{S}ekerka problem}, J. Comput. Phys., 127 (1996),
  pp.~246--267.

\end{thebibliography}
\end{document}